\documentclass[numbers,webpdf,imanum]{ima-authoring-template}%

\theoremstyle{thmstyletwo}%
\newtheorem{theorem}{Theorem}
\newtheorem{remark}{Remark}%
\newtheorem{lemma}{Lemma}%
\newcommand{\ds}{\displaystyle}

\numberwithin{equation}{section}

\begin{document}

\DOI{DOI HERE}
\copyrightyear{2022}
\vol{00}
\pubyear{2022}
\appnotes{Paper}
\copyrightstatement{Published by Oxford University Press on behalf of the Institute of Mathematics and its Applications. All rights reserved.}
\firstpage{1}


\title[Numerical analysis of constrained high-index saddle dynamics]{Discretization and index-robust error analysis for constrained high-index saddle dynamics on high-dimensional sphere}

\author{Lei Zhang
\address{\orgdiv{Beijing International Center for Mathematical Research, Center for Quantitative Biology, Center for Machine Learning Research}, \orgname{Peking
University}, \orgaddress{\street{Beijing}, \postcode{100871}, \country{China}}}}
\author{Pingwen Zhang
\address{\orgdiv{School of Mathematical Sciences, Laboratory of Mathematics and Applied Mathematics}, \orgname{Peking
University}, \orgaddress{\street{Beijing}, \postcode{100871}, \country{China}}}}
\author{Xiangcheng Zheng*
\address{\orgdiv{School of Mathematical Sciences}, \orgname{Peking
University}, \orgaddress{\street{Beijing}, \postcode{100871}, \country{China}}}}

\authormark{Zhang, Zhang, Zheng}

\corresp[*]{Corresponding author: \href{email:zhengxch@math.pku.edu.cn}{zhengxch@math.pku.edu.cn}}

\received{Date}{0}{Year}
\revised{Date}{0}{Year}
\accepted{Date}{0}{Year}


\abstract{We develop and analyze numerical discretization to the constrained high-index saddle dynamics, the dynamics searching for the high-index saddle points confined on the high-dimensional unit sphere. Compared with the saddle dynamics without constraints, the constrained high-index saddle dynamics has more complex dynamical forms, and additional operations such as the retraction and vector transport are required due to the constraint, which significantly complicate the numerical scheme and the corresponding numerical analysis. Furthermore, as the existing numerical analysis results usually depend on the index of the saddle points implicitly, the proved numerical accuracy may be reduced if the index is high in many applications, which indicates the lack of robustness with respect to the index. To address these issues, we derive the error estimates for numerical discretization of the constrained high-index saddle dynamics on high-dimensional sphere, and then improve it by providing an index-robust error analysis in an averaged norm by adjusting the relaxation parameters. The developed results provide mathematical supports for the accuracy of numerical computations.}


\maketitle

\section{Introduction}
 High-index saddle dynamics provides a powerful instrument in systematically finding the high-index saddle points and construction of the solution landscape \cite{Ben,EZho,YinPRL,YinSISC,YinSCM,ZhaSISC}. It has been applied to study various applications in physical and engineering problems \cite{HanXu,HanYin, han2021a, wang2021modeling,xu2021solution,YinPRL,yin2022solution, ZhangChe,ZhangCheDu,ZhaRen}. Here the index of saddle point refers to the Morse index characterized by the maximal dimension of a subspace on which its Hessian operator is negative definite \cite{Milnor}. If the additional nonlinear equality constraints exist, the constrained high-index saddle dynamics \cite{YinHua} was proposed for solving the constrained saddle point problems such as the Thomson problem \cite{Tho} and the Bose--Einstein Condensation \cite{bao2013mathematical}. However, due to the strong nonlinearities of the coupled equations in high-index saddle dynamics, rigorous numerical analysis are less considered compared with the growing numerical implementations in applications.  

There exist a couple of works investigating the convergence rate of the numerical saddle point to the target saddle point via sophisticated derivations \cite{All,CanLeg,Doye,EV2010,Farr,Gao,Gou,Hen2,Lev,Li,LiJi,Mehta,Qua,ZhaDu}.     
In contrast, conventional error analysis between the exact solutions of the continuous saddle dynamics and their numerical approximations in terms of the step size provides the dynamical convergence of numerical solutions to the saddle dynamics, which provides important physical information such as transition pathway and theoretically ensure the
accuracy of construction of the solution landscape \cite{HanYin,YinSCM,paper1}. In \cite{paper1}, a first-order numerical scheme for the high-index saddle dynamics without constraints was rigorously analyzed. However, due to the nonlinear equality constraint in the constrained high-index saddle dynamics, more complex dynamical form is encountered and additional operations in the numerical scheme such as the retraction and vector transport are required \cite{YinHua}, which significantly complicate the numerical scheme and the corresponding numerical analysis. Furthermore, the pointwise-in-time error estimate 
\begin{equation}\label{errpre}
\|x(t_n)-x_n\|+\sum_{i=1}^k\|v_i(t_n)-v_{i,n}\|\leq Q\tau
\end{equation}
was proved in \cite{paper1}, where $\{x_n,v_{1,n},\cdots,v_{k,n}\}$ serve as the numerical approximations of the state variable and corresponding eigenvectors $\{x(t),v_1(t),\cdots,v_k(t)\}$ at the time step $t_n$. However, the proof in \cite{paper1} indicates that the positive constant $Q$ depends on the index $k$ of saddle point. In some practical applications such as the Thomson problem, which considers the minimal-energy configuration of a group of classical charged particles confined on a sphere which interact with each other
via a Coulomb potential \cite{Sma,Tho}, the dimension of the energy function and the index $k$ of saddle point can be very high. In this case, the accuracy in (\ref{errpre}) could be reduced by the large factor $Q$.

Concerning the aforementioned issues, we aim to prove the error estimates of the numerical discretization to constrained high-index saddle dynamics on the high-dimensional unit sphere. The main difficulties we overcome lie in the complicated (nonlinear) forms of this dynamical system and the corresponding numerical scheme, which contains the retraction of the position $x$ and the transport and orthonormalization procedure of the vectors due to the sphere constraint. In particular, these features distinguishes the numerical analysis of saddle dynamics from those for ordinary differential equations. We then improve the analysis by developing an index-robust error estimate, i.e., the constant $Q$ in (\ref{errpre}) is independent from $k$, in an averaged norm by adjusting the relaxation parameters $\alpha$ and $\beta$. These results provide theoretical supports for the numerical accuracy of discretization of constrained high-index saddle dynamics and construction of solution landscapes for complex systems. 

The rest of the paper is organized as follows: In Section 2 we present formulations of the constrained high-index saddle dynamics and its numerical scheme. In Section 3 we prove auxiliary estimates, based on which we derive error estimates for the discretization of constrained high-index saddle dynamics in Section 4. Numerical experiments are performed in Section 5 to substantiate the theoretical findings. In Section 6, we prove index-robust error estimates for the numerical scheme in an averaged norm by adjusting the relaxation parameters. We finally address concluding remarks in the last section.

\section{Discretization of constrained high-index saddle dynamics}\label{sec2}
In this section we present the constrained high-index saddle dynamics on the high-dimensional unit sphere, as well as the corresponding numerical scheme.
\subsection{Formulation of constrained high-index saddle dynamics}
Given a $C^3$ function $E(x)$ with $x\in\mathbb R^d$ and define the corresponding natural force $F:\mathbb R^d\rightarrow \mathbb R^d$ and the negative Hessian $H\in \mathbb R^{d\times d}$ by $F(x)=-\nabla E(x)$ and $H(x)=-\nabla^2 E(x)$. It is clear that $H(x)=H(x)^\top$. Then the saddle dynamics for an index-$k$ saddle point of $E(x)$ with $1\leq k\in\mathbb N$ on the unit sphere $S^{d-1}$ reads \cite{YinHua}
\begin{equation}\label{sadk}
\left\{
\begin{array}{l}
\ds \frac{dx}{dt} =\alpha\bigg(I -xx^\top-2\sum_{j=1}^k v_jv_j^\top \bigg)F(x),\\[0.075in]
\ds \frac{dv_i}{dt}= \beta\bigg( I-xx^\top-v_iv_i^\top-2\sum_{j=1}^{i-1}v_jv_j^\top\bigg)H(x)v_i+\beta xv_i^\top F(x),~~1\leq i\leq k, 
\end{array}
\right.
\end{equation}
equipped with the initial conditions
$$x=X_0\in S^{d-1},~~v_i(0)=V_{i,0}\text{ such that }V_{i,0}^\top V_{j,0}=\delta_{ij},~~1\leq i,j\leq k. $$
Here $\alpha$ and $\beta$ are positive relaxation parameters. 
 In the original work \cite{YinHua}, these parameters are chosen as $\alpha=\beta=1$, and it was proved that the solutions $x$ and $\{v_i\}_{i=1}^k$ of (\ref{sadk}) satisfy
\begin{equation}\label{prop}
 x\in S^{d-1},~~v_i^\top x=0,~~v_{i}^\top v_{j}=\delta_{ij},~~1\leq i,j\leq k, t\geq 0,
 \end{equation} 
 and a linearly stable steady state of (\ref{sadk}) is an index-$k$ saddle point of $E(x)$.

Compared with the classical high-index saddle dynamics without constraints \cite{YinSISC},
\begin{equation}\label{sadsad}
\left\{
\begin{array}{l}
\ds \frac{dx}{dt} =\alpha\bigg(I -2\sum_{j=1}^k v_jv_j^\top \bigg)F(x),\\[0.075in]
\ds \frac{dv_i}{dt}= \beta\bigg( I-v_iv_i^\top-2\sum_{j=1}^{i-1}v_jv_j^\top\bigg)H(x)v_i,~~1\leq i\leq k,
\end{array}
\right.
\end{equation}
constrained high-index saddle dynamics (\ref{sadk}) exhibits a stronger coupling between $x(t)$ and $\{v_i(t)\}_{i=1}^k$ with several additional terms and has more properties (cf. (\ref{prop})) that we need to preserve in designing numerical methods, which distinguishes the current work from the existing numerical analysis for (\ref{sadsad}).

Throughout the paper we make the following regular assumptions on the force and the Hessian:

\noindent\textbf{Assumption A:} There exists constants $L_1,L_2\geq 0$ such that the following linearly growth and local Lipschitz conditions  
$$\begin{array}{c}
\ds \|H(x_2)-H(x_1)\|+\|F(x_2)-F(x_1)\|\leq L_1\|x_2-x_1\|,\\[0.1in]
\ds\|F(x)\|\leq L_1(1+\|x\|),~~x,x_1,x_2\in S^{d-1},
\end{array}  $$
as well as the boundedness of $F$ and $H$ on $S^{d-1}$
$$\max_{x\in S^{d-1}}\big(\|F(x)\|+\|H(x)\|\big)\leq L_2 $$
hold. Here $\|\cdot\|$ refers to the standard $l^2$ norm of a vector or a matrix.

\subsection{Numerical scheme}
We consider the numerical approximation of the constrained index-$k$ saddle dynamics (\ref{sadk}) for some $k\geq 1$ on the time interval $(0,T]$ for some $T>0$. Define a uniform partition $\{t_n\}_{n=0}^N$ of $[0,T]$ with the mesh size $\tau$. We discretize the first-order derivative by the explicit Euler scheme to get the reference equations for constrained index-$k$ saddle dynamics (\ref{sadk})  
\begin{equation}\label{Refk}
\left\{
\begin{array}{l}
\ds x(t_{n}) =x(t_{n-1})+\tau\alpha\bigg(I -x(t_{n-1})x(t_{n-1})^\top \\
\ds\hspace{0.7in}-2\sum_{j=1}^k v_j(t_{n-1})v_j(t_{n-1})^\top \bigg)F(x(t_{n-1}))+O(\tau^2),\\[0.075in]
\ds v_i(t_{n})=v_i(t_{n-1})+\tau\beta\bigg( I -x(t_{n-1})x(t_{n-1})^\top-v_i(t_{n-1})v_i(t_{n-1})^\top\\
\ds\hspace{0.5in}-2\sum_{j=1}^{i-1}v_j(t_{n-1})v_j(t_{n-1})^\top\bigg)H(x(t_{n-1}))v_i(t_{n-1})\\[0.2in]
\ds\hspace{0.7in}+\tau\beta x(t_{n-1})v_i(t_{n-1})^\top F(x(t_{n-1}))+O(\tau^2),~~1\leq i\leq k.
\end{array}
\right.
\end{equation}
Then we drop the truncation errors to obtain a first-order scheme of (\ref{sadk}) 
\begin{equation}\label{FDsadk}
\left\{
\begin{array}{l}
\ds \tilde x_{n} =x_{n-1}+\tau\alpha\bigg(I -x_{n-1}x_{n-1}^\top-2\sum_{j=1}^k v_{j,n-1}v_{j,n-1}^\top \bigg)F(x_{n-1}),\\
\ds x_n=\frac{\tilde x_n}{\|\tilde x_n\|},\\
\ds \tilde v_{i,n}=v_{i,n-1}+\tau\beta\bigg( I-x_{n-1}x_{n-1}^\top-v_{i,n-1}v_{i,n-1}^\top\\[0.15in]
\ds\hspace{1.2in}-2\sum_{j=1}^{i-1}v_{j,n-1}v_{j,n-1}^\top\bigg)H(x_{n-1})v_{i,n-1}\\[0.2in]
\ds\hspace{1.5in}+\tau\beta x_{n-1}v_{i,n-1}^\top F(x_{n-1}),~~1\leq i\leq k,\\[0.1in]
\hat v_{i,n}=\tilde v_{i,n}-\tilde v_{i,n}^\top x_n x_n,~~1\leq i\leq k,\\[0.05in]
\ds  v_{i,n}=\frac{1}{Y_{i,n}}\bigg(\ds\hat v_{i,n}-\sum_{j=1}^{i-1}(\hat v_{i,n}^\top v_{j,n})v_{j,n}\bigg),~~1\leq i\leq k
\end{array}
\right.
\end{equation}
for $1\leq n\leq N$ and
$$ x_0=X_0,~~v_{i,0}=V_{i,0},~~  Y_{i,n}:=\bigg(\|\hat v_{i,n}\|^2-\sum_{j=1}^{i-1}(\hat v_{i,n}^\top v_{j,n})^2\bigg)^{1/2},~~1\leq i\leq k.  $$
 
The second equation of (\ref{FDsadk}) represents the retraction in order to ensure that $x_n\in S^{d-1}$. The last two equations of (\ref{FDsadk}), which stand for the vector transport and the Gram-Schmidt orthonormalization procedure \cite{YinHua}, respectively, aim to ensure the last two properties of (\ref{prop}), that is,
\begin{equation}\label{propn}
v_{i,n}^\top x_n=0,~~v_{i,n}^\top v_{j,n}=\delta_{ij},~~1\leq i,j\leq k,~~0\leq n\leq N.
\end{equation}
Compared with the numerical schemes for index-$k$ saddle dynamics without constraints, see e.g. \cite{paper1}, the additional operations like the retraction and vector transport in (\ref{FDsadk}) caused from the sphere constraint make this scheme more complex and intensify the coupling of $x$ and $\{v_i\}_{i=1}^k$ that will significantly complicate the numerical analysis in subsequent sections.

\section{Auxiliary estimates}
We first prove several auxiliary estimates to support the error estimates.
\begin{lemma}\label{lem0k}
Under the Assumption A, the following estimate holds for $1\leq n\leq N$
$$\|x_n-\tilde x_n\|=\big|1-\|\tilde x_n\|\big|\leq 2\alpha^2 L_2^2 \tau^2. $$
\end{lemma}
\begin{proof}. We multiply $x_{n-1}^\top$ on both sides of the first equation of (\ref{FDsadk}) and use (\ref{propn}) to obtain
$$x_{n-1}^\top\tilde x_n=1. $$
We then multiply $\tilde x_{n}^\top$ on both sides of the first equation of (\ref{FDsadk}) and use $x_{n-1}^\top v_{j,n-1}=0$ for $1\leq j\leq k$ and the above equation to obtain
$$\|\tilde x_{n}\|^2=1+\tau\alpha\bigg(\tilde x_n^\top-x_{n-1}^\top -2\sum_{j=1}^k(\tilde x_n^\top-x_{n-1}^\top)v_{j,n-1}v_{j,n-1}^\top \bigg)F(x_{n-1}), $$
which, together with the Assumption A and the norm-preserving property of 
\begin{equation}\label{proj}
I-2\sum_{j=1}^kv_{j,n-1}v_{j,n-1}^\top ,
\end{equation}
 yields
$$\big|\|\tilde x_n\|^2-1\big|\leq \tau\alpha L_2\|\tilde x_n-x_{n-1}\|. $$
Similarly, by the first equation of (\ref{FDsadk}) again we have
$$\|\tilde x_n-x_{n-1}\|=\bigg\|\tau\alpha\bigg(I -x_{n-1}x_{n-1}^\top-2\sum_{j=1}^k v_{j,n-1}v_{j,n-1}^\top \bigg)F(x_{n-1})\bigg\|\leq 2\tau\alpha L_2. $$
Combining the above two equations yields
$$\big|\|\tilde x_n\|-1\big|\leq  \big|\|\tilde x_n\|^2-1\big|\leq 2\alpha^2 L_2^2 \tau^2,  $$
 and we incorporate this with 
$$\|x_n-\tilde x_n\|=\bigg\|\frac{\tilde x_n}{\|\tilde x_n\|}(1-\|\tilde x_n\|)\bigg\|= \big|1-\|\tilde x_n\|\big| $$
  to complete the proof.
\end{proof}
 \begin{lemma}\label{lem1k}
 Under the Assumption A, the following estimates hold for $1\leq n\leq N$ 
 $$\begin{array}{c}
 \ds \big|\tilde v_{m,n}^\top \tilde v_{i,n}\big|\leq Q\beta^2\tau^2,~~1\leq m<i\leq k;\\[0.1in]
 \ds  \big|\|\tilde v_{i,n}\|^2-1\big|\leq Q\beta^2\tau^2,~~1\leq i\leq k.
\end{array}  $$
Here the positive constant $Q$ is independent from $n$, $N$, $\tau$, $k$, $\alpha$ and $\beta$.
 \end{lemma}
\begin{proof}. To prove the first estimate, we apply the properties in (\ref{propn}) to calculate the product $\tilde v_{m,n}^\top \tilde v_{i,n}$ for $1\leq m< i\leq k$
\begin{equation*}
\begin{array}{l}
\ds \tilde v_{m,n}^\top \tilde v_{i,n}=\tau\beta\big(v_{m,n-1}^\top H(x_{n-1})v_{i,n-1}-2v_{m,n-1}^\top H(x_{n-1})v_{i,n-1}\\[0.05in]
\ds\qquad\qquad\qquad\qquad+v_{m,n-1}^\top H(x_{n-1})^\top v_{i,n-1}\big)+\tau^2\beta^2(\cdots)=\tau^2\beta^2(\cdots)
\end{array}
\end{equation*}
where we used the fact that first-order term is exactly zero by the symmetry of $H$. In order to specify the dependence of the estimates of the terms in $(\cdots)$ on $k$,
we observe that the only factor in the scheme of $\tilde v_{i,n}$ (or $\tilde v_{m,i}$) in (\ref{FDsadk}) that relates to $k$ is 
$$\sum_{j=1}^{i-1}v_{j,n-1}v_{j,n-1}^\top H(x_{n-1})v_{i,n-1}, $$
which could be considered as a projection of $H(x_{n-1})v_{i,n-1}$ with the norm less or equal to $\|H(x_{n-1})v_{i,n-1}\|\leq L_2$. By this means, the index $k$ is absorbed 
and the estimate of $(\cdots)$ is independent from $k$, which leads to 
\begin{equation}
\big|\tilde v_{m,n}^\top \tilde v_{i,n}\big|\leq Q\beta^2\tau^2,~~1\leq m<i\leq k.
\end{equation}

To estimate $\|\tilde v_{i,n}\|$, we multiply $v^\top_{i,n-1}$ on both sides of the third equation of (\ref{FDsadk}) and apply (\ref{propn}) to obtain for $1\leq i\leq k$
\begin{equation}\label{mh2k}
\ds v^\top_{i,n-1}\tilde v_{i,n}=1.
\end{equation}
We then multiply $x_{n-1}$ on both sides of the third equation of (\ref{FDsadk}) and apply (\ref{propn}) to obtain for $1\leq i\leq k$
\begin{equation}\label{xv}
 x_{n-1}^\top\tilde v_{i,n}=\tau\beta v_{i,n-1}^\top F(x_{n-1}),
 \end{equation}
 which implies
\begin{equation}\label{wh}
\tilde v_{i,n}^\top x_{n-1}x_{n-1}^\top=\tau\beta v_{i,n-1}^\top F(x_{n-1})x_{n-1}^\top.  \end{equation}
We finally multiply $\tilde v_{i,n}^\top$ on both sides of the third equation of (\ref{FDsadk}) and apply (\ref{propn}), (\ref{mh2k}), (\ref{xv}) and (\ref{wh}) to obtain
\begin{equation*}
\begin{array}{l}
\ds \tilde v_{i,n}^\top\tilde v_{i,n}=1+\tau\beta\bigg( \tilde v_{i,n}^\top-\tau\beta v_{i,n-1}^\top F(x_{n-1}) x_{n-1}^\top-v_{i,n-1}^\top\\[0.15in]
\ds\hspace{0.6in}-2\sum_{j=1}^{i-1}(\tilde v_{i,n}^\top-v_{i,n-1}^\top)v_{j,n-1}v_{j,n-1}^\top\bigg)H(x_{n-1})v_{i,n-1}\\[0.2in]
\ds\hspace{0.6in}+\tau^2\beta^2 \big(v_{i,n-1}^\top F(x_{n-1}) \big)^2,
\end{array}
\end{equation*}
which, together with
\begin{equation}\label{vv}
\|\tilde v_{i,n}-v_{i,n-1}\|^2=\tilde v_{i,n}^\top\tilde v_{i,n}-2\tilde v_{i,n}^\top v_{i,n-1}+1=\tilde v_{i,n}^\top\tilde v_{i,n}-1,
\end{equation}  
 implies
\begin{equation}\label{vvt} \|\tilde v_{i,n}-v_{i,n-1}\|^2\leq Q\tau\beta \|\tilde v_{i,n}-v_{i,n-1}\|+Q\tau^2\beta^2. 
\end{equation}
As this equation leads to 
$$\|\tilde v_{i,n}-v_{i,n-1}\|\leq \frac{Q\tau\beta+\sqrt{Q^2\tau^2\beta^2+4Q\tau^2\beta^2}}{2}\leq Q\beta\tau,$$
we incorporate this with (\ref{vv}) and (\ref{vvt}) to find that
$$\big|\tilde v_{i,n}^\top\tilde v_{i,n}-1\big|\leq Q\beta^2\tau^2,$$
which completes the proof.
\end{proof}

\begin{lemma}\label{lem15k}
 Suppose the Assumption A holds, $\alpha=\beta$ and $\sqrt{2}\beta L_2 \tau\leq 1-\theta$ for some $0<\theta<1$, then the following estimates hold for $1\leq n\leq N$ 
 $$\begin{array}{c}
\ds \|\hat v_{i,n}-\tilde v_{i,n}\|\leq Q\beta^2\tau^2,~~1\leq i\leq k;\\[0.1in]
 \ds \big|\hat v_{m,n}^\top \hat v_{i,n}\big|\leq Q(\beta^2+\beta^4)\tau^2,~~1\leq m<i\leq k;\\[0.1in]
 \ds  \big|\|\hat v_{i,n}\|^2-1\big|\leq Q(\beta^2+\beta^4)\tau^2,~~1\leq i\leq k.
\end{array}  $$
Here the positive constant $Q$ is independent from $n$, $N$, $\tau$, $k$, $\alpha$ and $\beta$.
 \end{lemma}
 \begin{proof}.
 By the fourth equation of (\ref{FDsadk}) we have
 $$\|\hat v_{i,n}-\tilde v_{i,n}\|=\big|\tilde v_{i,n}^\top x_n\big|. $$
 By the first and the third equations in (\ref{FDsadk}), direct calculations show that
 $$\begin{array}{l}
 \ds\big|\tilde v_{i,n}^\top x_n\big|=\frac{1}{\|\tilde x_n\|}\big|\tau\big(\beta v_{i,n-1}^\top F(x_{n-1})+\alpha F(x_{n-1})^\top v_{i,n-1}\\[0.15in]
 \ds\qquad\qquad\qquad-2\alpha v_{i,n-1}^\top F(x_{n-1})\big)+\tau^2\alpha\beta(\cdots)\big|.
\end{array}   $$
Note that the first-order term in this equation is exact zero if $\alpha=\beta$. By Lemma \ref{lem0k} we have 
$$\frac{1}{\|\tilde x_n\|}\leq \frac{1}{1-2\alpha^2 L_2^2 \tau^2}\leq \frac{1}{1-(1-\theta)^2}\leq \frac{1}{\theta}.$$ Thus we obtain from the above equations that
\begin{equation} \label{v-v}
\|\hat v_{i,n}-\tilde v_{i,n}\|=\big|\tilde v_{i,n}^\top x_n\big|\leq Q\beta^2\tau^2. 
\end{equation}
 
 For $1\leq m\leq i\leq k$ we get
 $$\hat v_{m,n}^\top \hat v_{i,n}=\tilde v_{m,n}^\top \tilde v_{i,n}-x_n^\top \tilde v_{i,n}x_n^\top \tilde v_{m,n}. $$
 If $m<i$, we apply Lemma \ref{lem1k} and (\ref{v-v}) to obtain
 $$\big|\hat v_{m,n}^\top \hat v_{i,n}\big|\leq Q(\beta^2+\beta^4)\tau^2.$$
 If $m=i$ we again employ Lemma \ref{lem1k} and (\ref{v-v}) to get
$$\big|\|\hat v_{i,n}\|^2-1\big|=\big|\| \tilde v_{i,n}\|^2-1-(x_n^\top \tilde v_{i,n})^2\big|\leq Q(\beta^2+\beta^4)\tau^2. $$ 
Thus we complete the proof.
 \end{proof}
\begin{lemma}\label{lem2k}
Suppose the Assumption A holds, $\alpha=\beta$ and $\sqrt{2}\beta L_2 \tau\leq 1-\theta$ for some $0<\theta<1$, the following estimate holds for $1\leq n\leq N$ and $\tau$ sufficiently small 
\begin{equation}\label{lem2keq}
 \|v_{i,n}-\hat v_{i,n}\|\leq Q\tau^2,~~1\leq i\leq k.
 \end{equation}
Here the positive constant $Q$ is independent from $n$, $N$ and $\tau$ but may depend on $k$, $\alpha$ and $\beta$. 
\end{lemma}
\begin{proof}.
Based on Lemma \ref{lem15k}, the proof could be performed by the same techniques as \cite[Lemma 4.2]{paper1} and is thus omitted.
 \end{proof}

\section{Index-dependent error estimate}
We present error estimates for the scheme (\ref{FDsadk}). Note that in this section, the constant $Q$ in the main theorem may depend on the index $k$, which implies that the corresponding estimate is not index-robust. An improvement for analyzing an index-robust error estimate in an averaged norm is given in Section \ref{secind}.

We bound the errors 
$$e^x_n:=x(t_n)-x_n,~~e^{v_i}_{n}:=v_i(t_n)-v_{i,n},~~1\leq n\leq N,~~1\leq i\leq k.$$ 
 based on the following lemma.
 \begin{lemma}\label{lemgron}
 Suppose there exists a non-negative sequence $\{Z_n\}_{n=0}^N$ with $Z_0=0$ satisfying
 \begin{equation}\label{ZZ} Z_n\leq Z_{n-1}+A\tau Z_{n-1}+B\tau^2\sum_{m=1}^{n-1}Z_m+C\tau^2 \end{equation}
 for $1\leq n\leq N$ and for some non-negative constants $A$, $B$ and $C$ with $A+BT\neq 0$. Then the following estimate of $Z_n$ holds
 \begin{equation*}
 Z_n\leq \frac{C}{A+BT}\big(e^{(A+BT)T}-1\big)\tau,~~1\leq n\leq N.
  \end{equation*}
 \end{lemma}
 \begin{proof}.
 We first prove by induction that 
 \begin{equation}\label{Z} Z_n\leq \frac{C}{A+BT}\big((1+A\tau+BT\tau)^n-1\big)\tau,~~1\leq n\leq N. \end{equation}
 It is clear that (\ref{Z}) holds for $n=1$. Suppose (\ref{Z}) holds for $1\leq n\leq \tilde n-1$. Then we invoke these estimates in (\ref{ZZ}) to obtain
 $$\begin{array}{l}
 \ds Z_{\tilde n}\leq \frac{C}{A+BT}\big((1+A\tau+BT\tau)^{\tilde n-1}-1\big)\tau \big(1+A\tau+B\tau^2(\tilde n-1)\big)+C\tau^2 \\[0.15in]
 \ds\quad~\leq \frac{C}{A+BT}\big((1+A\tau+BT\tau)^{\tilde n}-(1+A\tau+B\tau T)\big)\tau +C\tau^2\\[0.15in]
 \ds\quad~=\frac{C}{A+BT}\big((1+A\tau+BT\tau)^{\tilde n}-1\big)\tau-\frac{C}{A+BT}(A\tau+B\tau T)\tau +C\tau^2\\[0.15in]
 \ds\quad~=\frac{C}{A+BT}\big((1+A\tau+BT\tau)^{\tilde n}-1\big)\tau.
\end{array}  $$
That is, (\ref{Z}) holds for $n=\tilde n$ and thus for any $1\leq n\leq N$ by mathematical induction. We then apply the following estimate
$$(1+A\tau+BT\tau)^n=\big(1+(A+BT)\tau\big)^{\frac{1}{(A+BT)\tau}(A+BT)T}\leq e^{(A+BT)T} $$
to complete the proof.
 \end{proof}

\begin{theorem}\label{thmevk}
Suppose the Assumption A holds, $\alpha=\beta$ and $\sqrt{2}\beta L_2 \tau\leq 1-\theta$ for some $0<\theta<1$. Then the following estimate holds 
$$\|e^x_{n}\|+\mathcal E_n\leq Q\tau,~~1\leq n\leq N,~~\mathcal E_m:=\sum_{j=1}^k\|e^{v_j}_{m}\|. $$
Here $Q$ is independent from $\tau$, $n$ and $N$ but may depend on $k$, $\alpha$ and $\beta$.
\end{theorem}
\begin{proof}. 
We subtract the first equation of (\ref{Refk}) from that of (\ref{FDsadk}) to obtain
\begin{equation}\label{ee}
 \begin{array}{l}
\ds\hspace{-0.1in} e^x_{n} =e^x_{n-1}+\tau\alpha\bigg(I -x(t_{n-1})x(t_{n-1})^\top\\
\ds\hspace{1in}-2\sum_{j=1}^k v_j(t_{n-1})v_j(t_{n-1})^\top \bigg)F(x(t_{n-1}))\\
\ds\hspace{0.7in} -\tau\alpha\bigg(I -x_{n-1}x_{n-1}^\top-2\sum_{j=1}^k v_{j,n-1}v_{j,n-1}^\top \bigg)F(x_{n-1})\\[0.2in]
\ds\hspace{1in}+O(\tau^2)+(\tilde x_n-x_n)\\[0.05in]
\ds\hspace{-0.2in}\qquad=e^x_{n-1}+\tau\alpha\big(F(x(t_{n-1}))-F(x_{n-1})\big)\\[0.1in]
\ds\hspace{-0.1in}\qquad\quad-\tau \alpha\big(x(t_{n-1})x(t_{n-1})^\top F(x(t_{n-1}))-x_{n-1}x_{n-1}^\top F(x_{n-1})\big)\\[0.05in]
\ds\hspace{-0.1in}\qquad\quad-2\tau\alpha\sum_{j=1}^k\big(v_j(t_{n-1})v_j(t_{n-1})^\top F(x(t_{n-1}))\\
\ds\hspace{1in}-v_{j,n-1}v_{j,n-1}^\top F(x_{n-1})\big)+O(\tau^2)+(\tilde x_n-x_n). 
\end{array} 
\end{equation}
We bound the forth term on the right-hand side of (\ref{ee}) as
\begin{equation}\label{proj1} \begin{array}{l}
\ds \bigg\|\tau\sum_{j=1}^k\big(v_j(t_{n-1})v_j(t_{n-1})^\top F(x(t_{n-1}))-v_{j,n-1}v_{j,n-1}^\top F(x_{n-1})\big)\bigg\|\\[0.15in]
\ds\quad\leq \tau\sum_{j=1}^k\big\|e^{v_j}_{n-1}v_j(t_{n-1})^\top F(x(t_{n-1}))+v_{j,n-1}(e^{v_j}_{n-1})^\top F(x(t_{n-1}))\\[0.2in]
\ds\qquad+v_{j,n-1} v_{j,n-1}^\top (F(x(t_{n-1}))-F(x_{n-1}))\big\|\leq Q\tau \big(\mathcal E_{n-1}+\|e^x_{n-1}\| \big).
\end{array} \end{equation}
Here we used the projection nature of the third difference (i.e. the summation containing $F(x(t_{n-1}))-F(x_{n-1})$) such that $Q$ is independent from $k$. The other right-hand side terms could be bounded in a similar manner. We thus obtain from (\ref{ee}) and Lemma \ref{lem0k} that
\begin{equation}\label{hl}
\|e^x_n\|\leq \|e^x_{n-1}\|+Q\tau \|e^x_{n-1}\|+Q\tau\mathcal E_{n-1}+Q\tau^2. 
\end{equation}
An application of the standard discrete Gronwall inequality yields
\begin{equation}\label{MH}
\|e^x_{n}\|\leq Q\tau\sum_{m=1}^{n-1}\mathcal E_m +Q\tau ,~~1\leq n\leq N. 
\end{equation}

We then subtract the third equation of (\ref{Refk}) from that of (\ref{FDsadk}) and split
$v_i(t_n)-\tilde v_{i,n}$ as 
$$v_i(t_n)-\tilde v_{i,n}=e^{v_i}_n+(v_{i,n}-\hat v_{i,n})+(\hat v_{i,n}-\tilde v_{i,n})$$
 to obtain
\begin{equation}\label{ev}
\begin{array}{rl}
\ds e^{v_i}_{n}&\ds\hspace{-0.1in}=e^{v_i}_{n-1}+\tau\beta\big(H(x(t_{n-1}))v_i(t_{n-1})-H(x_{n-1})v_{i,n-1}\big)\\[0.05in]
&\ds\quad~~-\tau\beta\big[ v_i(t_{n-1})v_i(t_{n-1})^\top H(x(t_{n-1}))v_i(t_{n-1})\\[0.05in]
&\ds\qquad\quad-v_{i,n-1}v_{i,n-1}^\top H(x_{n-1})v_{i,n-1}\big]\\[0.05in]
&\ds\quad~~-\tau\beta \big[ x(t_{n-1})x(t_{n-1})^\top H(x(t_{n-1}))v_i(t_{n-1})\\[0.05in]
&\ds\qquad\quad-x_{n-1}x_{n-1}^\top H(x_{n-1})v_{i,n-1}\big]\\[0.05in]
&\ds\quad~~-2\tau\beta \sum_{j=1}^{i-1}\big[ v_j(t_{n-1})v_j(t_{n-1})^\top H(x(t_{n-1}))v_i(t_{n-1})\\[0.05in]
&\ds\qquad\quad-v_{j,n-1}v_{j,n-1}^\top H(x_{n-1})v_{i,n-1}\big]\\[0.05in]
&\ds\quad~~ +\tau\beta\big[ x(t_{n-1})v_i(t_{n-1})^\top F(x(t_{n-1}))-x_{n-1}v_{i,n-1}^\top F(x_{n-1})\big]\\[0.05in]
&\ds\quad~~ -(v_{i,n}-\hat v_{i,n})-(\hat v_{i,n}-\tilde v_{i,n})+O(\tau^2) .
\end{array}
\end{equation}
We apply Lemmas \ref{lem15k} and \ref{lem2k} and similar derivations as (\ref{ee}) to get
\begin{equation*}
\ds\| e^{v_i}_{n}\|\leq \|e^{v_i}_{n-1}\|+Q\tau\big(\|e^x_{n-1}\|+\|e^{v_i}_{n-1}\|\big)+Q\tau\mathcal E_{n-1} +Q\tau^2.
\end{equation*}
Adding this equation from $i=1$ to $k$ and using (\ref{MH}) we obtain
\begin{equation}\label{mh8k}
\ds \mathcal E_{n}\leq \mathcal E_{n-1}+Q\tau \mathcal E_{n-1}+Q\tau^2\sum_{m=1}^{n-1} \mathcal E_{m} +Q\tau^2.
\end{equation}
Then an application of Lemma \ref{lemgron} leads to the conclusion of the theorem.
\end{proof}

\section{Numerical experiments}
In this section, we carry out numerical experiments to test the numerical accuracy of the scheme (\ref{FDsadk}). For practical applications of the constrained high-index saddle dynamics (\ref{sadk}) we refer \cite{YinHua} for more details. Let $T=\alpha=\beta=1$, $\|e^x\|:=\max_{1\leq n\leq N}\|x(t_n)-x_n\|$,  $\|e^{v_i}\|:=\max_{1\leq n\leq N}\|v_i(t_n)-v_{i,n}\|$ for $1\leq i\leq k$
and we use numerical solutions computed under $\tau=2^{-13}$ as the reference solutions due to the unavailability of the exact solutions.

\subsection{Accuracy test}
We consider the index-1 constrained saddle dynamics for 
the 4-well potential proposed in \cite{Col}
\begin{equation}\label{col}
 E(x_1,x_2)= x_1^4-px_1^2+x_2^4-x_2^2+qx_1^2x_2^2
\end{equation}
with the initial conditions
$$X_0=\frac{1}{\sqrt{2}}(1,1),~~V_1=\frac{1}{\sqrt{2}}(-1,1) $$ 
and the parameters 
$$(\text{i})~(p,q)=(5,1),~~(\text{ii})~(p,q)=(10,5). $$
Numerical results are presented in Tables \ref{table1:1}-\ref{table2:1}, which demonstrate the first-order accuracy of the scheme (\ref{FDsadk}) as proved in Theorem \ref{thmevk}.
\begin{table}[h]
\setlength{\abovecaptionskip}{0pt}
\centering
\caption{Convergence rates for surface (\ref{col}) under the parameter (i).}
\begin{tabular}{ccccc} \cline{1-5}
$\tau$& $\|e^x\|$ & conv. rate &  $\|e^{v_1}\|$ &conv. rate\\ \cline{1-5}		
$1/2^5$&	2.31E-02&		&2.31E-02&	\\
$1/2^6$&	1.15E-02&	1.01&	1.15E-02&	1.02\\
$1/2^7$&	5.67E-03&	1.02&	5.67E-03&	1.02\\
$1/2^8$&	2.79E-03&	1.02&	2.79E-03&	1.03\\
				\hline
			\end{tabular}
			\label{table1:1}
		\end{table}

\begin{table}[h]
\setlength{\abovecaptionskip}{0pt}
\centering
\caption{Convergence rates for surface (\ref{col}) under the parameter (ii).}
\begin{tabular}{ccccc} \cline{1-5}
$\tau$& $\|e^x\|$ & conv. rate & $\|e^{v_1}\|$ &conv. rate\\ \cline{1-5}		
$1/2^5$&	5.40E-02&		&5.40E-02	&\\
$1/2^6$&	2.63E-02&	1.04&	2.63E-02&	1.04\\
$1/2^7$&	1.29E-02&	1.02&	1.29E-02&	1.02\\
$1/2^8$&	6.31E-03&	1.03&	6.31E-03&	1.03\\
				\hline
			\end{tabular}
			\label{table2:1}
		\end{table}

\subsection{Dynamic convergence}\label{sec62}
We consider the index-1 constrained saddle dynamics for the Rosenbrock type function
\begin{equation}\label{Ros} E(x_1,x_2,x_3)=a(\sqrt{3}x_2-3x_1^2)^2+b(\sqrt{3}x_1-1)^2+a(\sqrt{3}x_3-3x_2^2)^2+b(\sqrt{3}x_2-1)^2 
\end{equation}
equipped with the parameters $a=2$, $b=-9.8$, $T=5$ and different initial conditions
$$X_0=\frac{1}{\sqrt{29}}(2,-3,4),~~X'_0=\frac{1}{\sqrt{3}}(-1,-1,1),~~V_{1,0}=\frac{1}{\sqrt{2}}(1,1,0). $$
Under the given parameters, this surface has an index-1 saddle point located at
$$\bigg(\frac{1}{\sqrt{3}},\frac{1}{\sqrt{3}},\frac{1}{\sqrt{3}}\bigg).$$
 Numerical results are presented in Figure \ref{fig2}, which shows that the numerical saddle dynamics could reach the target saddle point under different initial values, even though the curvature of the Rosenbrock type function is complicated. Furthermore, the dynamic convergence of the numerical saddle dynamics as the step size $\tau$ decreases is observed, which demonstrates that
the proposed scheme is appropriate for computing the dynamic pathways
for constructing the solution landscapes, e.g. \cite{HanYin,Yin2020nucleation,YinSCM,YinHua}.

	\begin{figure}[h!]
	\setlength{\abovecaptionskip}{0pt}
	\centering	
\,	\includegraphics[width=3in,height=2.7in]{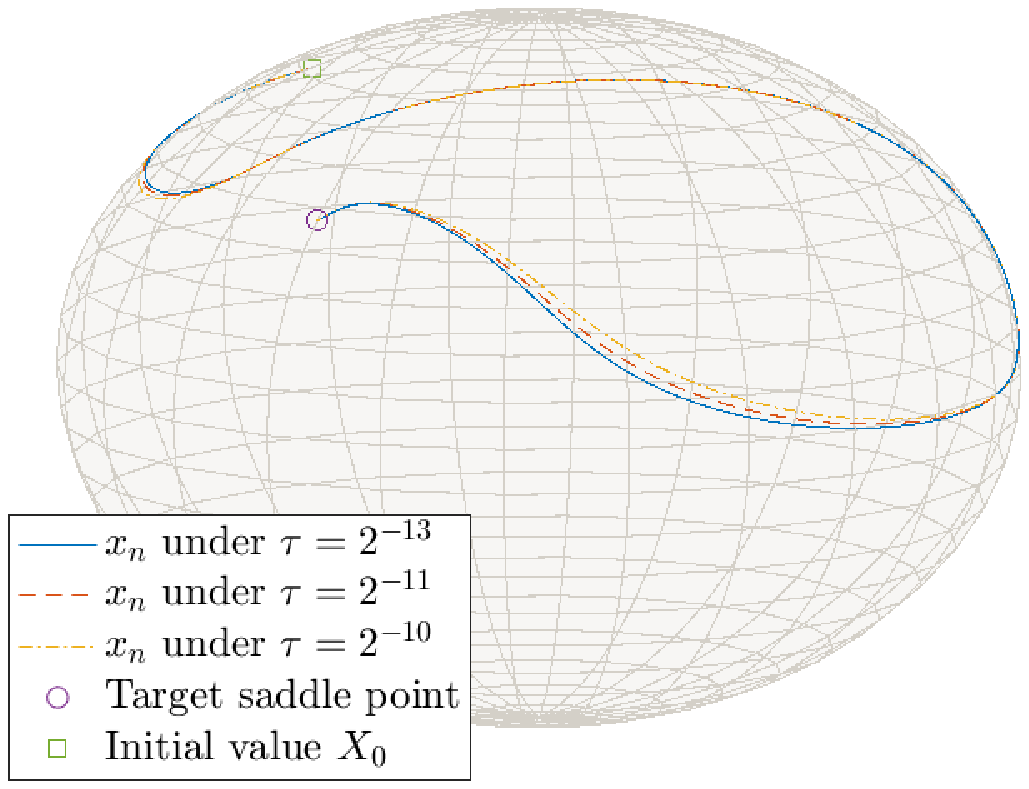}\hspace{-0.4in}
\includegraphics[width=3in,height=2.7in]{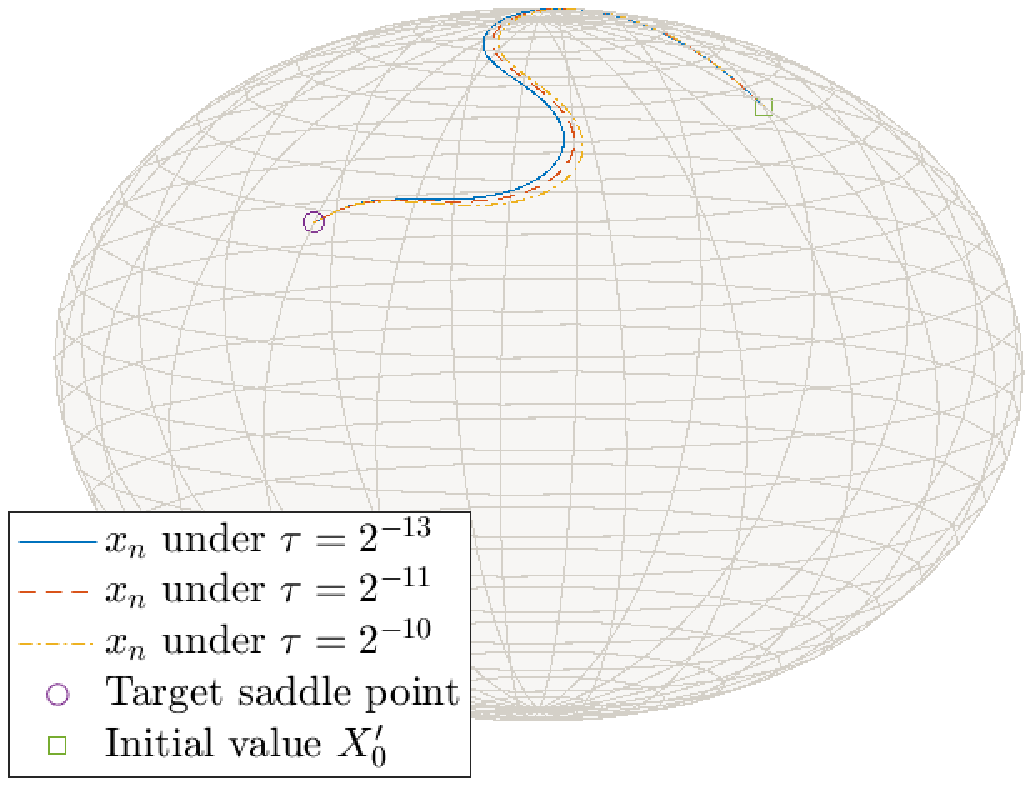}
	\caption{Pathway convergence of constrained saddle dynamics under different initial values.}
	\label{fig2}
\end{figure}

\section{Index-robust numerical analysis}\label{secind}
In previous sections, we develop error estimates for the numerical discretization (\ref{FDsadk}) of the constrained high-index saddle dynamics (\ref{sadk}). However, as the constant $Q$ in the estimate of Theorem \ref{thmevk} depends on $k$, the numerical accuracy could be reduced if $k$ is large enough in real applications. In other words, the estimate in Theorem \ref{thmevk} is not index-robust. Therefore, we aim to develop an index-robust error estimate in this section to eliminate the dependence of $k$ on $Q$ in Theorem \ref{thmevk}, which serves as an improvement for the previous results. 

The key ideas to achieve this goal lie in adjusting the relaxation parameter and estimating the errors in the averaged norm. Suppose   
\begin{equation}\label{betacond}
\alpha=\beta\leq \frac{Q_0}{k} 
 \end{equation}
for some positive constant $Q_0$. Furthermore, the following averaged norm $\|\cdot\|_A$ is defined for any sequence of vectors $\{g_i\}_{i=1}^k$
$$\|g\|_A:=\frac{1}{k}\sum_{i=1}^k \|g_i\|.$$
Based on these presumptions, we immediately obtain from Lemmas \ref{lem0k}--\ref{lem15k} that for $1\leq n\leq N$
\begin{equation}\label{lemimp}
\begin{array}{c}
\ds \|x_n-\tilde x_n\|\leq M\frac{\tau^2}{k^2},~~\|\hat v_{i,n}-\tilde v_{i,n}\|\leq M\frac{\tau^2}{k^2},~~1\leq i\leq k,\\[0.1in]
 \ds \big|\hat v_{m,n}^\top \hat v_{i,n}\big|\leq M\frac{\tau^2}{k^2},~~1\leq m<i\leq k,\\[0.1in]
 \ds  \big|\|\hat v_{i,n}\|^2-1\big|\leq M\frac{\tau^2}{k^2},~~1\leq i\leq k.
\end{array}
\end{equation}
Here the constant $M$ is independent from $\tau$, $n$, $N$, $\alpha$, $\beta$ and $k$. To prove the desired results, the Lemma \ref{lem2k} and the Theorem \ref{thmevk} need to be re-estimated in the following subsections.
\subsection{Re-estimation of Lemma \ref{lem2k}}
We re-estimate Lemma \ref{lem2k} such that the constant $Q$ in (\ref{lem2keq}) is independent from $k$. Though the derivations of the re-estimation follow the ideas of the proof of Lemma \ref{lem2k} in \cite[Lemma 4.2]{paper1}, there are essential differences in that the condition (\ref{betacond}) is invoked to eliminate the dependence of the constants $Q$ on $k$ throughout the derivations.
\begin{lemma}\label{lem2knew}
Suppose the Assumption A and (\ref{betacond}) hold, then the following estimate holds for $1\leq n\leq N$ and for $\tau$ small enough  
$$\|v_{i,n}-\hat v_{i,n}\|\leq Q\tau^2,~~1\leq i\leq k.$$
Here the positive constant $Q$ is independent from $n$, $N$, $\tau$, $\alpha$, $\beta$ and $k$. 
\end{lemma}
\begin{proof}.
Let $G>M$ be a fixed positive number where $M$ is defined in (\ref{lemimp}). Then we firstly prove that, if $\tau$ satisfies the constraint
\begin{equation}\label{aimh}
\frac{M+\tau^2G^2}{\ds\big(1-M\tau^2-G^2\tau^4\big)^{1/2}}\leq G,
\end{equation}
then the following estimates hold
\begin{equation}\label{leme3}
|\hat v_{i,n}^\top v_{m,n}|\leq G\frac{\tau^2}{k^2},~~1\leq m<i\leq k.
\end{equation}
Note that the condition (\ref{aimh}) is independent from $k$ and is valid if $\tau$ is sufficiently small. We prove this argument by induction on the subscription $m$.
For $m=1$ we apply (\ref{lemimp}), the last equation of (\ref{FDsadk}) as well as (\ref{aimh}) to obtain for $1<i\leq k$
\begin{equation}\label{leme1}
\ds |\hat v_{i,n}^\top v_{1,n}|=\frac{|\hat v_{i,n}^\top \hat v_{1,n}|}{(\|\hat v_{1,n}\|^2)^{1/2}}\leq \frac{M\tau^2/k^2}{(1-M\tau^2/k^2)^{1/2}}\leq\frac{M}{(1-M\tau^2)^{1/2}}\frac{\tau^2}{k^2}\leq G\frac{\tau^2}{k^2}. 
\end{equation}
Thus, (\ref{leme3}) holds with $m=1$. Suppose (\ref{leme3}) holds for $1\leq m<m^*$ for some $1\leq m^*<k-1$. Then we invoke (\ref{leme3}) with $1\leq m<m^*$ and (\ref{lemimp}) into the expression of $\hat v_{i,n}^\top v_{m^*,n}$ to obtain for $m^*<i\leq k$
$$\begin{array}{rl}
\ds | \hat v_{i,n}^\top v_{m^*,n}|&\hspace{-0.1in}\ds=\frac{\ds\bigg|\hat v_{i,n}^\top\hat v_{m^*,n}-\sum_{j=1}^{m^*-1}(\hat v_{m^*,n}^\top v_{j,n})(\hat v_{i,n}^\top v_{j,n})\bigg|}{\ds\bigg(\|\hat v_{m^*,n}\|^2-\sum_{j=1}^{m^*-1}(\hat v_{m^*,n}^\top v_{j,n})^2\bigg)^{1/2}} \\[0.15in]
&\ds\hspace{-0.1in}\leq\frac{\ds M\frac{\tau^2}{k^2}+(m^*-1)\frac{G^2\tau^4}{k^4}}{\ds \bigg(1-M\frac{\tau^2}{k^2}-(m^*-1)\frac{G^2\tau^4}{k^4}\bigg)^{1/2}} \\[0.2in]
&\ds\hspace{-0.1in} \leq \frac{M+G^2\tau^2}{\big(1-M\tau^2-\tau^4G^2\big)^{1/2}}\frac{\tau^2}{k^2}\leq G\frac{\tau^2}{k^2},~~m^*<i\leq k. 
\end{array}$$
That is, (\ref{leme3}) holds for $m=m^*$ and thus holds for any $1\leq m<k$ by mathematical induction.

Similar to the above derivations, there exists a constant $Q>0$ independent from  $n$, $N$, $\tau$, $\alpha$, $\beta$ and $k$ such that for $1\leq i\leq k$ and $1\leq n\leq N$ 
 \begin{equation}\label{bndY}
 \ds (1-Q\tau^2)^{1/2}\leq |Y_{i,n}|\leq (1+Q\tau^2)^{1/2}.
\end{equation}
Note that this implies
\begin{equation}\label{bndY2}
|1-Y_{i,n}|\leq |1-Y_{i,n}|(1+Y_{i,n})=|Y^2_{i,n}-1|\leq Q\tau^2.
\end{equation}    
Then we remain to estimate $v_{i,n}-\hat v_{i,n}$ for $1\leq i\leq k$. According to the definition of $v_{i,n}$ we have
\begin{equation*}
v_{i,n}-\hat v_{i,n}=\frac{1}{Y_{i,n}}\bigg((1-Y_{i,n})\hat v_{i,n}-\sum_{j=1}^{i-1}(\hat v_{i,n}^\top v_{j,n})v_{j,n}\bigg),
\end{equation*}
which, together with (\ref{lemimp}), (\ref{leme3}), (\ref{bndY}) and (\ref{bndY2}), implies
\begin{equation*}
\begin{array}{rl}
\ds \|v_{i,n}-\hat v_{i,n}\|&\hspace{-0.1in}\ds\leq \frac{1}{|Y_{i,n}|}\bigg(|1-Y_{i,n}|\|\hat v_{i,n}\|+\sum_{j=1}^{i-1}|\hat v_{i,n}^\top v_{j,n}|\bigg)\\
&\hspace{-0.1in}\ds\leq Q|1-Y_{i,n}|+Qk\frac{\tau^2}{k^2}\leq Q\tau^2.
\end{array}
\end{equation*}
Thus we complete the proof. 
\end{proof}
\begin{remark}
From the proof of Lemma \ref{lem2knew} we note that the same conclusion could be obtained under a weaker constraint on $\beta$ than (\ref{betacond}), that is, the power of $k$ on the denominator could be less than 1. However, we will show in the next proof that (\ref{betacond}) is necessary. 
\end{remark}
\subsection{Index-robust error estimate}
We prove an index-robust error estimate for the numerical scheme (\ref{FDsadk}) in the averaged norm $\|\cdot\|_A$ to the constrained high-index saddle dynamics (\ref{sadk}) in the following theorem.
\begin{theorem}\label{thmevkz}
Suppose the Assumption A and (\ref{betacond}) hold and $\sqrt{2}\beta L_2 \tau\leq 1-\theta$ for some $0<\theta<1$, then the following estimate holds for $\tau$ sufficiently small 
$$\|e^x_{n}\|+\|e_n^v\|_A\leq Q\tau,~~1\leq n\leq N,~~\|e_n^v\|_A=\frac{1}{k}\sum_{j=1}^k\|e^{v_j}_{n}\|=\frac{1}{k}\mathcal E_n. $$
Here $Q$ is independent from $\tau$, $n$, $N$, $\alpha$, $\beta$ and $k$.
\end{theorem}
\begin{proof}. 
For the estimate of $e^x_n$, we observe from (\ref{ee}) that only its fourth right-hand side term could result in the dependence of $Q$ on $k$. As this term was carefully estimated in (\ref{proj1}), in which the constant $Q$ is independent from $k$, we could apply similar derivations as those for (\ref{MH}) in the proof of Theorem \ref{thmevk} and the condition (\ref{betacond}) to obtain
\begin{equation}\label{MHz}
\|e^x_{n}\|\leq Q\frac{\tau}{k}\sum_{m=1}^{n-1}\mathcal E_m +Q\tau= Q\tau\sum_{m=1}^{n-1}\|e^v_m\|_A +Q\tau ,~~1\leq n\leq N. 
\end{equation}

For the estimate of $e_n^{v_i}$ from (\ref{ev}), it suffices to pay attention on the difference of two summations
\begin{equation}
\begin{array}{rl}
&\ds \tau\beta \bigg\|\sum_{j=1}^{i-1}\big[ v_j(t_{n-1})v_j(t_{n-1})^\top H(x(t_{n-1}))v_i(t_{n-1})\\[0.05in]
&\ds\qquad\quad-v_{j,n-1}v_{j,n-1}^\top H(x_{n-1})v_{i,n-1}\big]\bigg\|\\
&\ds\qquad\leq\tau\beta\bigg\|\sum_{j=1}^{i-1}\big[e^{v_j}_{n-1}v_j(t_{n-1})^\top H(x(t_{n-1}))v_i(t_{n-1})\\[0.175in]
&\ds\qquad\quad+v_{j,n-1}(e^{v_j}_{n-1})^\top H(x(t_{n-1}))v_i(t_{n-1})\\[0.1in]
&\ds\qquad\quad+v_{j,n-1}v_{j,n-1}^\top (H(x(t_{n-1}))-H(x_{n-1}))v_i(t_{n-1})\\[0.1in]
&\ds\qquad\quad+v_{j,n-1}v_{j,n-1}^\top H(x_{n-1})e^{v_i}_{n-1}\big]  \bigg\|\\
&\ds\qquad\leq \frac{Q\tau}{k}\big(\mathcal E_{n-1}+\|e^x_{n-1}\|+\|e^{v_i}_{n-1}\|\big).
\end{array}
\end{equation}
Here the constant $Q$ is independent from $\alpha$, $\beta$ and $k$.
Then similar estimates as those for (\ref{hl})  yield
\begin{equation*}
\ds\| e^{v_i}_{n}\|\leq \|e^{v_i}_{n-1}\|+Q\frac{\tau}{k}\big(\|e^x_{n-1}\|+\|e^{v_i}_{n-1}\|\big)+Q\frac{\tau}{k}\mathcal E_{n-1} +Q\tau^2.
\end{equation*}
Adding this equation from $i=1$ to $k$ and multiplying the resulting equation by $k^{-1}$ lead to
$$\|e^v_n\|_A\leq \|e^v_{n-1}\|_A+Q\frac{\tau}{k}\|e^x_{n-1}\|+Q\frac{\tau}{k}\|e^v_{n-1}\|_A+Q\tau\|e^v_{n-1}\|_A+Q\tau^2. $$
 We invoke (\ref{MHz}) to obtain
$$\|e^v_n\|_A\leq \|e^v_{n-1}\|_A+Q\frac{\tau^2}{k}\sum_{m=1}^{n-1}\|e^v_m\|_A+Q\tau\|e^v_{n-1}\|_A+Q\tau^2. $$
Then an application of Lemma \ref{lemgron} leads to the conclusion of the theorem.
\end{proof}

\section{Concluding remarks}
In this paper we prove error estimates for numerical approximation to constrained high-index saddle dynamics, which extends the existing results for the classical high-index saddle dynamics without constraints in the literature. We then improve the proved results by developing an index-robust error estimate for the proposed numerical scheme in an averaged norm by adjusting the relaxation parameters $\alpha$ and $\beta$. These results provide theoretical supports for the numerical accuracy of discretization of constrained high-index saddle dynamics. 

It is worth mentioning that the current work mainly focuses on the constrained high-index saddle dynamics on the finite interval $t\in [0,T]$ and prove error estimates to show the dynamic convergence of numerical solutions. For most applications, the target saddle point $x_*$ could be reached within a certain number of iterations, which corresponds to a finite $T$ as imposed in the current work. Nevertheless, the convergence to the saddle point is theoretically determined by the rate of $x(t)\rightarrow x_*$ as $t\rightarrow \infty$, which suggests the error estimate of $x_n-x_*$ for $n\rightarrow \infty$ as those for optimization algorithms. In a very recent work \cite{Luo} such convergence analysis for the high-index saddle dynamics without constraints is performed. However, due to the stronger coupling between $x$ and $\{v_i\}_{i=1}^k$ in the constrained case, it is not straightforward to extend the methods in \cite{Luo} to analyze the convergence of $x_n-x_*$ for numerical methods to  the constrained high-index saddle dynamics (\ref{sadk}) and we will investigate this interesting topic in the near future.

\section*{Acknowledgements}
This work was partially supported by the National Natural Science Foundation of China No.~12050002 and 21790340, the the National Key Research and Development Program of China No. 2021YFF1200500, the  International Postdoctoral Exchange Fellowship Program (Talent-Introduction Program) No.~YJ20210019, and the China Postdoctoral Science Foundation No.~2021TQ0017 and 2021M700244.



\begin{thebibliography}{}
\bibitem{All} E. L. Allgower and K. Georg, {\it Introduction to numerical continuation methods}. SIAM, 2003.

\bibitem{bao2013mathematical} W. Bao and Y. Cai, {Mathematical theory and numerical methods for {B}ose--{E}instein condensation}. {\it Kinet. Relat. Models} 6 (2013), 1--135.

\bibitem{Ben} M. Benzi, G. Golub, J. Liesen, Numerical solution of saddle point problems. {\it Acta Numer.} 14 (2005), 1--137.


\bibitem{CanLeg} E. Cances, F. Legoll, M.C. Marinica, K. Minoukadeh, F. Willaime, Some improvements of the activation-relaxation technique method for finding transition pathways on potential energy surfaces. {\it J. Chem. Phys.} 130 (2009) 114711.


\bibitem{Col} P. Collins, G. Ezra, S. Wiggins,  Index k saddles and dividing surfaces in
phase space with applications to isomerization dynamics. {\it J. Chem. Phys.} 134 (2011), 244105.

\bibitem{Doye} J. Doye and D. Wales, Saddle points and dynamics of Lennard-Jones clusters, solids, and supercooled liquids. {\it J Chem Phys} 116 (2002), 3777--3788.

\bibitem{EV2010}
{ W. E, E. Vanden-Eijnden}, 
{Transition-path theory and path-finding algorithms for the study of rare events}, 
{\it Annu. Rev. Phys. Chem.}, 61 (2010), 391-420.

\bibitem{EZho} W. E and X. Zhou, The gentlest ascent dynamics. {\it Nonlinearity} 24 (2011), 1831--1842.


\bibitem{Farr} P. E. Farrell, \'{A}. Birkisson, and S. W. Funke, Deflation Techniques for Finding Distinct Solutions of Nonlinear Partial Differential Equations. {\it SIAM J. Sci. Comput.} 37 (2015), A2026--A2045.

\bibitem{Gao} W. Gao, J. Leng, and X. Zhou, An iterative minimization formulation for saddle point search. {\it SIAM J. Numer. Anal.} 53 (2015), 1786--1805.

\bibitem{Gou} N. Gould, C. Ortner and D. Packwood, A dimer-type saddle search algorithm with preconditioning and linesearch. {\it Math. Comp.} 85 (2016), 2939--2966.

\bibitem{HanXu} Y. Han, Z. Xu, A. Shi, L. Zhang, Pathways connecting two opposed bilayers with a fusion pore: a molecularly-informed phase field approach. {\it Soft Matter}, 16 (2020), 366--374.

\bibitem{HanYin} Y. Han, J. Yin, P. Zhang, A. Majumdar, L. Zhang, Solution landscape of a reduced Landau--de Gennes model on a hexagon. {\it Nonlinearity} 34 (2021), 2048--2069.

\bibitem{han2021a} Y. Han, J. Yin, Y. Hu, A. Majumdar, L. Zhang, 
 {Solution landscapes of the simplified Ericksen-Leslie model and its comparison with the reduced Landau-de Gennes model}.
 {\it Proc. Roy. Soc. A-Math. Phy.}, 477 (2021), 20210458.

%
\bibitem{Hen2} G. Henkelman, B. Uberuaga, and H. J\'onsson, A climbing image nudged elastic band method for finding saddle points and minimum energy paths. {\it J. Chem. Phys.} 113 (2000), 9901--9904.

\bibitem{Lev} A. Levitt and C. Ortner, Convergence and cycling in walker-type saddle search algorithms. {\it SIAM J. Numer. Anal.} 55 (2017), 2204--2227.

\bibitem{Li} Y. Li and J. Zhou, A minimax method for finding multiple critical points and its applications to semilinear PDEs. {\it SIAM J. Sci. Comput.} 23 (2001), 840--865.

\bibitem{LiJi} Z. Li, J. Zhou, A local minimax method using virtual geometric objects: part II-for finding equality constrained saddles. {\it J. Sci. Comput.} 78 (2019), 226--245.

\bibitem{Luo} Y. Luo, X. Zheng, X. Cheng, L. Zhang, Convergence analysis of discrete high-index saddle dynamics. arXiv:2204.00171.

\bibitem{Mehta} D. Mehta, Finding all the stationary points of a potential-energy landscape via numerical polynomial-homotopy-continuation method, {\em Phys Rev E} 84 (2011) 025702. 

\bibitem{Milnor} J. W. Milnor, {\em Morse Theory}, Princeton University Press, 1963.


\bibitem{Qua} W. Quapp and J. M. Bofill, Locating saddle points of any index on potential energy surfaces by the generalized gentlest ascent dynamics. {\it Theor. Chem. Accounts} 133 (2014), 1510.

\bibitem{Sma} S. Smale, Mathematical problems for the next century. {\it Math. Intell.} 20 (1998) 7--15.

\bibitem{Tho} J. J. Thomson, XXIV. On the structure of the atom: an investigation of the stability and periods of oscillation of a number of corpuscles arranged at equal intervals around the circumference of a circle; with application of the results to the theory of atomic structure,
London, Edinburgh, Dublin Phil. Mag. J. Sci. 7 (1904) 237--265.

\bibitem{wang2021modeling} W. Wang, L. Zhang, P. Zhang, Modelling and computation of liquid crystals. {\it Acta Numerica}
30 (2021), 765--851.


\bibitem{xu2021solution} Z. Xu, Y. Han, J. Yin, B. Yu, Y. Nishiura, L. Zhang, 
  {Solution landscapes of the diblock copolymer-homopolymer model under two-dimensional confinement}. {\it Phys. Rev. E}, 104 (2021), 014505.

\bibitem{YinHua} J. Yin, Z. Huang, L. Zhang, Constrained high-index saddle dynamics for the solution landscape with equality constraints. {\it J. Sci. Comput.} 91 (2022), 62.
 
 \bibitem{Yin2020nucleation} J. Yin, K. Jiang, A.-C. Shi, P. Zhang, L. Zhang, {Transition pathways connecting crystals and quasicrystals}. {\it Proc. Natl. Acad. Sci. U.S.A.} 118 (2021), e2106230118.

\bibitem{YinPRL} J. Yin, Y. Wang, J. Chen, P. Zhang, L. Zhang, Construction of a pathway map on a complicated energy landscape. {\it Phys. Rev. Lett.} 124 (2020), 090601.

\bibitem{YinSISC} J. Yin, L. Zhang, P. Zhang, High-index optimization-based shrinking dimer method for finding high-index saddle points. {\it SIAM J. Sci. Comput.} 41 (2019), A3576--A3595.

\bibitem{YinSCM} J. Yin, B. Yu, L. Zhang, Searching the solution landscape by generalized high-index saddle dynamics. {\it Sci. China Math.} 64 (2021), 1801.



 \bibitem{yin2022solution} J. Yin, L. Zhang, P. Zhang, {Solution landscape of the Onsager model identifies non-axisymmetric critical points}.
{\it Physica D} 430 (2022), 133081.
     
\bibitem{ZhaDu} J. Zhang, Q. Du, Shrinking dimer dynamics and its applications to saddle point search. {\it SIAM J. Numer. Anal.} 50 (2012), 1899--1921.

\bibitem{ZhangChe} L. Zhang, L. Chen, Q. Du, Morphology of critical nuclei in solid-state phase transformations. {\it Phys. Rev. Lett.} 98 (2007), 265703.

\bibitem{ZhangCheDu} L. Zhang, L. Chen, Q. Du, Simultaneous prediction of morphologies of a critical nucleus and an equilibrium precipitate in solids. {\it Commun. Comput. Phys.} 7 (2010), 674--682.

\bibitem{ZhaSISC} L. Zhang, Q. Du, Z. Zheng, Optimization-based shrinking dimer method for finding transition states. {\it SIAM J. Sci. Comput.} 38 (2016), A528--A544.

\bibitem{ZhaRen} L. Zhang, W. Ren, A. Samanta, and Q. Du, Recent developments in computational modelling of nucleation in phase transformations. {\it NPJ Comput. Materials} 2 (2016), 16003.

\bibitem{paper1} L. Zhang, P. Zhang, X. Zheng, Error estimates of Euler discretization to high-index saddle dynamics. arXiv: 2111.05156.


\end{thebibliography}
\end{document}